\documentclass[11pt,twoside]{amsart}

\usepackage[
backend = biber,
citestyle = alphabetic,
style = alphabetic,
maxnames = 50,
giveninits=true,
]{biblatex}
\usepackage{amsmath,amsfonts,amssymb,amsthm}
\usepackage[all,cmtip]{xy}
\usepackage{tikz-cd}

\usepackage[colorlinks]{hyperref}
\hypersetup{
    colorlinks,
    citecolor=magenta,
    filecolor=red,
    linkcolor=blue,
    urlcolor=purple,
    pdftitle = Counterexamples to the Kuznetsov--Shinder L-equivalence conjecture
}
\usepackage{mathtools}

\usepackage{cleveref}
\usepackage[T2A]{fontenc}
\DeclareSymbolFont{cyrillic}{T2A}{cmr}{m}{n}
\DeclareMathSymbol{\Sha}{\mathalpha}{cyrillic}{216}
\usepackage{xcolor}
\renewcommand\thefootnote{\textcolor{red}{\arabic{footnote}}} 
\usepackage[normalem]{ulem}

\numberwithin{equation}{section}
\usepackage{microtype}
\newtheorem{dummy}{dummy}[section]

\newtheorem{theorem}[dummy]{Theorem}
\newtheorem{corollary}[dummy]{Corollary}
\newtheorem{lemma}[dummy]{Lemma}

\theoremstyle{definition}
\newtheorem{definition}[dummy]{Definition}
\newtheorem{remark}[dummy]{Remark}

\newtheorem{question}[dummy]{Question}
\theoremstyle{plain}

\newtheorem{conjecture}[dummy]{Conjecture}

\newcommand{\A}{\ensuremath{\mathbb{A}}}
\newcommand{\Z}{\ensuremath{\mathbb{Z}}}

\newcommand{\Q}{\ensuremath{\mathbb{Q}}}
\newcommand{\CC}{\ensuremath{\mathbb{C}}}

\newcommand{\LL}{\ensuremath{\mathbb{L}}}

\newcommand{\klat}{\ensuremath{\Lambda_{\operatorname{K3}}}}
\newcommand{\extklat}{\ensuremath{\tilde{\Lambda}_{\operatorname{K3}}}}
\newcommand{\mmlatx}{\ensuremath{\tilde{\Lambda}_{X}}}
\newcommand{\crs}{\ensuremath{\operatorname{crs}}}

\def\Db{\calD^{b}}

\newcommand{\gravk}{\ensuremath{K_0(\operatorname{Var}_k)}}
\newcommand{\gravc}{\ensuremath{K_0(\operatorname{Var}_\CC)}}
\newcommand{\gghs}{\ensuremath{K_0^{\oplus}(\operatorname{HS}_{\Z,2})}}
\newcommand{\gghsf}{\ensuremath{K_0^{\oplus}(\operatorname{HS}_{\Z})}}

\def\disc{\operatorname{disc}}

\def\Kn{{\ensuremath{\operatorname{K3}^{[n]}}}}

\def\N{\operatorname{N}}

\def\dv{\operatorname{div}}

\DeclareMathOperator{\rk}{rk}

\DeclareMathOperator{\Hom}{Hom}

\DeclareMathOperator{\olPic}{\overline{Pic}}
\DeclareMathOperator{\Br}{Br}

\DeclareMathOperator{\SHom}{\mathcal{H\kern -1pt}\textit{om}} 
\DeclareMathOperator{\SEnd}{\mathcal{E\kern -1pt}\textit{nd}} 
\DeclareMathOperator{\SExt}{\mathcal{E\kern -1pt}\textit{xt}} 

\DeclareMathOperator{\NS}{NS}

\def\dar[#1]{\ar@<2pt>[#1]\ar@<-2pt>[#1]}

\newcommand\calD{\mathcal{D}}

\newcommand{\functionstar}[4]{
\begin{array}{rcl} #1 &\longrightarrow &#2 \\ #3&\longmapsto &#4 \end{array}
}

\title{Counterexamples to the Kuznetsov--Shinder L-equivalence conjecture}

\author[R.~Meinsma]{Reinder Meinsma}
\address{
Fakultät für Mathematik und Informatik,
Universität des Saarlandes,
Campus E2.4, 66123 Saarbrücken, Germany
}
\email{meinsma@math.uni-sb.de}

\begin{document}
\begin{abstract}
     We disprove a conjecture of Kuznetsov--Shinder, which posits that $D$-equivalent simply connected varieties are $L$-equivalent, by constructing a counterexample using moduli spaces of sheaves on K3 surfaces. 
\end{abstract}
\maketitle
\tableofcontents

\let\thefootnote\relax\footnotetext{The research in this work has been supported by the MIS (MIS/BEJ - F.4545.21) Grant from the FNRS, a mobility grant by the FNRS and an ACR Grant from the Universit\'e Libre de Bruxelles.}

\section{Introduction}
The Grothendieck ring $\gravk$ of varieties over a field $k$ is the free abelian group generated by isomorphism classes of varieties over $k$ modulo the scissor relations:
\[
[X] = [X\setminus Z] - [Z] \qquad \text{for every closed } Z\subset X.
\]
The product in $\gravk$ is given by the fibre product: $$[X]\cdot [Y] \coloneqq [X\times_kY].$$ 
The class of the affine line, also called the Lefschetz motive, is denoted by $$\LL\coloneqq [\A^1_k]\in \gravk.$$ 

The class of a variety $X$ in the Grothendieck ring of varieties is an important motivic invariant of $X$. Among other things, the class $[X]$ carries information used to study birational equivalence problems \cite{LL03,NS19} and the bounded derived category $\Db(X)$ of coherent sheaves on $X$ \cite{Kaw02,KS18}. 

In \cite{Bor17}, Borisov proved that the Lefschetz motive $\LL$ is a zero-divisor in $\gravc$. This led to the difficult problem of computing the kernel of the morphism $\gravc\to \gravc[\LL^{-1}]$, which resulted in the definition of $L$-equivalence by Kuznetsov and Shinder \cite[\S 1.2]{KS18}:
\begin{definition} Let $X$ and $Y$ be two varieties over $k$.
    \begin{enumerate}
        \item We say that $X$ and $Y$ are \textit{$L$-equivalent} if there is a positive integer $r\in \Z_{>0}$ such that 
        \[
        \LL^r\cdot ([X]-[Y]) = 0 \in \gravk.
        \]
        \item We say that $X$ and $Y$ are \textit{$D$-equivalent} if there is a $k$-linear, exact equivalence of their bounded derived categories of coherent sheaves $\Db(X)\simeq \Db(Y)$. 
    \end{enumerate}
\end{definition}

A key question in this area is the following conjecture by Kuznetsov--Shinder:
\begin{conjecture}\cite[Conjecture 1.6]{KS18} \label{conj: D implies L}
    Let $X$ and $Y$ be smooth projective simply connected varieties. If $X$ and $Y$ are $D$-equivalent, then $X$ and $Y$ are $L$-equivalent:
    \[
    \Db(X)\simeq \Db(Y)\implies X\sim_LY.
    \]
\end{conjecture}
We note that \cite{KS18} and \cite{IMOU20} initially conjectured independently that all $L$-equivalent smooth projective varieties must be $D$-equivalent, including non-simply connected ones. However, counterexamples were found by Efimov \cite{Efi18}, who proved that there exist pairs of abelian varieties that are $D$-equivalent but not $L$-equivalent. The conjecture was revised in \cite{KS18} to apply specifically to simply connected varieties. Conversely, the modification in \cite{IMOU20} reformulated the conjecture to require the $L$-equivalence equation within a specialized Grothendieck group of varieties, see \cite[Problem 7.2]{IMOU20}, which is partially proved in \cite{CLP23}.

Conjecture \ref{conj: D implies L} has been studied in many examples. For $L$-equivalent K3 surfaces, one can find examples in \cite{IMOU20,KS18,HL18,KKM20,SZ20}. In all these cases, the $L$-equivalent K3 surfaces are also derived equivalent. In \cite{Mei24}, building on the work of \cite{Efi18}, it was proved that very general $L$-equivalent K3 surfaces are $D$-equivalent, which is a proof of a special case of the converse of Conjecture \ref{conj: D implies L}. In the higher-dimensional setting, it was proved in \cite{Oka21} that if $S$ and $M$ are $D$-equivalent and $L$-equivalent K3 surfaces, then the Hilbert schemes $S^{[n]}$ and $M^{[n]}$ are also $D$-equivalent and $L$-equivalent, but not birational for certain values of $n$. We note that $S^{[n]}$ and $M^{[n]}$ can both be realised as \textit{fine} moduli spaces of stable sheaves on $S$.

The main result of this paper is the following theorem, which proves that for any $n\geq 2$, there exist complex hyperk\"ahler manifolds of \Kn-type that are $D$-equivalent but not $L$-equivalent. This disproves Conjecture \ref{conj: D implies L}.
\begin{theorem}[See Theorem \ref{thm: D equivalence does not imply L equivalence Main Text}] \label{thm: D equivalence does not imply L equivalence Intro}
     Let $(S,H)$ be a primitively polarised complex K3 surface of degree $H^2 = 2g-2 \geq 2$, and assume that $S$ has Picard rank 1. Then the moduli spaces of stable sheaves $\olPic^0\coloneqq M(0,H,1-g)$ and $\olPic^{g-1} \coloneqq M(0,H,0)$ are $D$-equivalent but not $L$-equivalent.
\end{theorem}
The derived equivalence of $\olPic^0$ and $\olPic^{g-1}$ was proved in \cite{ADM16}. Our contribution is proving that they are not $L$-equivalent. 

\subsection*{Main ingredients of the proof of Theorem \ref{thm: D equivalence does not imply L equivalence Intro}}
A powerful tool in studying $L$-equivalence systematically is by using additive invariants that behave well under $L$-equivalence. By \textit{additive invariant}, we mean a group homomorphism $\gravk \to G$ for some group $G$.
The most important example for this paper is the group homomorphism 
\begin{equation}\label{eq:group homomorphism to hodge structures}
    \functionstar{\gravc}{\gghsf}{\text{ }[X]}{[H^\bullet(X,\Z)].}
\end{equation}

Here $\gghsf$ denotes the Grothendieck group of integral, graded, pure, polarisable Hodge structures. 
This group homomorphism has the important property that $[H^\bullet(X,\Z)] = [H^\bullet(Y,\Z)]$ in $\gghsf$ whenever $X$ and $Y$ are $L$-equivalent smooth projective varieties \cite{Efi18}. 
This fact was used in \cite{Efi18} to prove that derived equivalent abelian varieties need not be $L$-equivalent, and in \cite{Mei24} to prove that very general $L$-equivalent K3 surfaces are derived equivalent. 

In this paper, we introduce two birational invariants of a hyperk\"ahler manifold of \Kn-type $X$, which we call the \textit{gluing group}, denoted $G(X)$, and the \textit{coarseness of $X$}, which is a positive integer denoted $\crs(X)$. 

The gluing group of $X$ is defined to be the group 
\[
G(X)\coloneqq \frac{H^2(X,\Z)}{\NS(X)\oplus T(X)},
\] where $T(X)$ is the transcendental lattice of $X$. We prove that the order of this group is an additive invariant. Moreover, $L$-equivalent hyperk\"ahler manifolds $X$ and $Y$ satisfy $|G(X)| = |G(Y)|$. 

We define the coarseness $\crs(X)$ for any hyperk\"ahler manifold $X$ of \Kn-type to be the divisibility of a generator for the orthogonal complement of $H^2(X,\Z)$ in the Markman--Mukai lattice $\mmlatx$ (see Definition \ref{def: coarseness}). In the case when $X$ is a moduli space of sheaves on a K3 surface $S$, 
the coarseness of $X$ is equal to the order of the Brauer class $\alpha_X\in \Br(X)$ that obstructs the existence of a universal sheaf as described in \cite{Cal00,MM24}, see Lemma \ref{lem: basic results about coarseness}. Therefore, $X$ is a fine moduli space of sheaves on $S$ if and only if the coarseness of $X$ is 1.

Our main result linking the gluing group and coarseness to $L$-equivalence is the following:
\begin{theorem}[See Theorem \ref{thm: L equivalent same divisibility} and Corollary \ref{cor: L equivalent moduli spaces have same divisibility}] \label{thm: L equivalent same divisibility Intro}
    Let $X$ be a hyperk\"ahler manifold of \Kn-type, where $n\geq 2$. The positive integer 
    \[\frac{\disc(T(X))}{\crs(X)}\] equals $|G(X)|$ and therefore is invariant under $L$-equivalence. 
    In particular, if $M$ and $M'$ are $L$-equivalent moduli spaces of sheaves on the same K3 surface, then the orders of the obstruction Brauer classes $\alpha_M\in \Br(M)$ and $\alpha_{M'}\in \Br(M')$ are equal. 
\end{theorem}
Theorem \ref{thm: D equivalence does not imply L equivalence Intro} is an application of Theorem \ref{thm: L equivalent same divisibility Intro}. 

\subsection*{Open problems}

We propose the following conjecture, which is a partial converse of Conjecture \ref{conj: D implies L}.

\begin{conjecture} \label{conj: L implies D (may still be true)}
    Let $X$ and $Y$ be projective hyperk\"ahler manifolds. Assume that $[X]-[Y]\in \gravk$ is annihilated by a power of $\LL$, i.e. $X$ and $Y$ are $L$-equivalent. Then the bounded derived categories of coherent sheaves of $X$ and $Y$ are equivalent:
    \[
    X\sim_LY\implies \Db(X)\simeq \Db(Y).
    \]
\end{conjecture}

We already discussed evidence for Conjecture \ref{conj: L implies D (may still be true)} above. All currently known examples of $L$-equivalent K3 surfaces and hyperk\"ahler manifolds are $D$-equivalent as well \cite{IMOU20,KS18,HL18,KKM20,SZ20}. Moreover, Conjecture \ref{conj: L implies D (may still be true)} was established in \cite{Mei24} when $X$ and $Y$ are very general K3 surfaces. The results of this paper also support Conjecture \ref{conj: L implies D (may still be true)}. For example, Corollary \ref{cor: L equivalence preserves fineness} asserts that if $M$ and $M'$ are $L$-equivalent moduli spaces of stable sheaves on some K3 surface $S$, then $M$ is a fine moduli space if and only if $M'$ is. Thus, in this case we also have an equivalence $\Db(M)\simeq \Db(M')$ by a recent result of Zhang \cite[Corollary 1.5]{Zha25}.

Conjecture \ref{conj: L implies D (may still be true)} has some interesting implications for $L$-equivalence of moduli spaces of sheaves on K3 surfaces. For example by \cite[Corollary 9.6]{Bec22}, it predicts a positive answer to the following question:
\begin{question}
    Suppose $X$ is a hyperk\"ahler manifold of \Kn-type which is $L$-equivalent to a moduli space of stable sheaves $M$ on a K3 surface $S$. Is $X$ itself isomorphic to a moduli space of stable sheaves on $S$?
\end{question}

Additionally, Conjecture \ref{conj: L implies D (may still be true)} predicts the existence of a Hodge isometry $T(X)\simeq T(Y)$ for any pair of $L$-equivalent hyperk\"ahler manifolds $X$ and $Y$ of \Kn-type by \cite[Corollary 9.3]{Bec22}. In particular, $\disc(T(X))$ is predicted to be invariant under $L$-equivalence, and therefore we should also have $\crs(X) = \crs(Y)$. 

\subsection*{Acknowledgements} 
I would like to thank Evgeny Shinder for our discussions that never fail to inspire me, and for his useful comments on an earlier draft of this paper. I also want to thank Felix K\"ung, Hannah Dell, Dominique Mattei, Corey Brooke, Sarah Frei, and Lisa Marquand for useful discussions and interest in my work, and Hsueh-Yung Lin, Christopher Nicol and Wahei Hara for pointing out counterexamples to an earlier version of Conjecture \ref{conj: L implies D (may still be true)}.

\section{The Grothendieck group of Hodge structures}
We recall some important facts about Grothendieck groups of Hodge structures. Our main reference is \cite{Efi18}.

We denote by $\gghsf$ the Grothendieck group of  Hodge structures. That is, $\gghsf$ is the free abelian group generated by isomorphism classes of integral, graded, pure, polarisable Hodge structures, subject to the relations 
\[
[H_1\oplus H_2] = [H_1] + [H_2] \qquad \text{for all } H_1,H_2\in \operatorname{HS}_\Z. 
\]
Similarly, we denote by $\gghs$ the Grothendieck group of integral, pure, polarisable Hodge structures of weight 2. We have the natural group homomorphism 
\[
\functionstar{\gghsf}{\gghs}{\text{ }[H^\bullet]}{[H^2].}
\]

\subsection{Transcendental Hodge substructures and gluing groups}
Given a pure, integral, polarisable Hodge structure $H\in \operatorname{HS}_{\Z,2}$ of weight 2, we denote its transcendental Hodge substructure by $T(H)$. Explicitly, $T(H)$ is defined to be the minimal integral primitive Hodge substructure of $H$ such that $T(H)_\CC$ contains $H^{2,0}$. 
That is, 
\[
T(H) = \Big(\bigcap_{V\subset H_\Q}V\Big) \cap H,
\]
where the intersection runs over all rational Hodge substructures $V$ of $H_\Q$ such that $V^{2,0} = H^{2,0}$. Since $H_\Q$ is a finite-dimensional $\Q$-vector space, we have $T(H)_\CC = \bigcap_{V}V_\CC$, from which it follows that $T(H)$ inherits a Hodge structure from $H$ by setting $T(H)^{p,q} = \bigcap_V V^{p,q}$.

Equivalently, we may define the transcendental Hodge substructure of $H$ as follows. The rational Hodge structure $H_\Q$ decomposes into a direct sum of simple, polarisable, rational Hodge structures $H_\Q \simeq \bigoplus_{i = 1}^n V_i$, since the category of rational, pure, polarisable Hodge structures of weight 2 is semi-simple, see for example \cite[\S II.7, Lemma 7.26]{Voi02}. Then we define $T_\Q$ to be the direct sum of those $V_i$ for which we have $V_i\cap H^{2,0} \neq 0$, and finally we put $T\coloneqq T_\Q\cap H$.  
This description of $T(H)$ makes it clear that for any two Hodge structures $H_1,H_2\in \operatorname{HS}_{\Z,2}$, we have $T(H_1\oplus H_2) = T(H_1)\oplus T(H_2)$. Moreover, any morphism of Hodge structures $f\colon H_1\to H_2$ restricts to a morphism of transcendental Hodge substructures $T(H_1)\to T(H_2)$. Therefore, taking the transcendental Hodge substructure defines an additive functor $\operatorname{HS}_{\Z,2}\to \operatorname{HS}_{\Z,2}$ and thus a group homomorphism 
\[
\functionstar{\gghs}{\gghs}{\text{ }[H]}{[T(H)].}
\]

We denote $\NS(H)\coloneqq H^{1,1}\cap H$. After fixing a polarisation $\psi$ on $H$, the orthogonal complement $\NS(H)^\perp$ with respect to $\psi$ is a Hodge substructure of $H$ for which we have $H^{2,0} = (\NS(H)^\perp)^{2,0}$. This shows that $\NS(H)\cap T(H) = 0$, provided $H$ is polarisable. On the other hand, if $H$ is non-polarisable, it is possible that $\NS(H)$ and $T(H)$ intersect non-trivially \cite[Example 3.3.2]{Huy16}.

For a smooth, projective variety $X$, we write \[T(X)\coloneqq T(H^2(X,\Z)).\] Since \eqref{eq:group homomorphism to hodge structures} maps $L$-equivalent smooth projective varieties to the same element of $\gghsf$, it follows that $L$-equivalent smooth projective varieties $X$ and $Y$ must satisfy $[T(X)] = [T(Y)]$ in $\gghs$. 

If $X$ is a K3 surface or a hyperk\"ahler manifold of \Kn-type, then $T(X) = \NS(X)^\perp \subset H^2(X,\Z)$ \cite[Lemma 3.3.1]{Huy16}, and thus $\NS(X)\oplus T(X)\subset H^2(X,\Z)$ is a subgroup of finite index. 

\begin{definition} \label{def: gluing group}
    For an integral, polarisable Hodge structure $H$ of weight 2, we denote the \textit{gluing group} of $H$ by
    \[G(H)\coloneqq \frac{H}{\NS(H) \oplus T(H)},\] 
    where $T(H)$ is the transcendental Hodge substructure of $H$, and $\NS(H) \coloneqq H^{1,1}\cap H$. When $X$ is a smooth projective variety, we write $$G(X) \coloneqq G(H^2(X,\Z)).$$
\end{definition}
If $H_1,H_2\in \operatorname{HS}_{\Z,2},$ then we have $G(H_1\oplus H_2)\simeq G(H_1)\oplus G(H_2).$ 
Combined with \eqref{eq:group homomorphism to hodge structures}, this proves the following key lemma (see also \cite[\S 4]{Mei24}):
\begin{lemma} \label{lem: gluing subgroup order is L-invariant}
    We have a well-defined group homomorphism 
    \begin{equation} \label{eq: gluing group homomorphism}
    \functionstar{\gghs}{\Q^\times}{\text{ }[H]}{|G(H)_{\text{tors}}|.}
    \end{equation}
    Moreover, if $X$ and $Y$ are $L$-equivalent hyperk\"ahler manifolds, then $G(X)$ and $G(Y)$ are finite, and we have $$|G(X)| = |G(Y)|.$$
\end{lemma}
\begin{proof}
    Let $H_1,H_2\in \operatorname{HS}_{\Z,2}$. Then, since we have $G(H_1\oplus H_2)\simeq G(H_1)\oplus G(H_2),$ it follows that $G(H_1\oplus H_2)_\mathrm{tors} \simeq G(H_1)_\mathrm{tors} \oplus G(H_2)_\mathrm{tors}$. This shows that the map \eqref{eq: gluing group homomorphism} is well-defined. Moreover, it is clearly a group homomorphism. 
\end{proof}
The group homomorphism in Lemma \ref{lem: gluing subgroup order is L-invariant} factors through the Grothendieck group $K_0^\oplus(\Z)$ of the category of finitely generated abelian groups, with relations given by split exact sequences. The homomorphism 
\[
\functionstar{K_0^\oplus(\Z)}{\Q^\times}{A}{|A_{\text{tors}}|}
\]
was also used in \cite{BGLL21}. It is important to note that the relations in $\gghs$ are given by direct sums, and not by short exact sequences, see \cite[Definition 6.9 and Lemma 6.12]{BGLL21}. 

\subsection{Hodge lattices of K3-type} Our main reference for lattices is \cite{Nik80}.
A pure, effective, integral, polarisable Hodge structure $H$ of weight 2 is said to be \textit{of K3-type} if $H^{2,0}\simeq \CC$ \cite[Definition 3.2.3]{Huy16}.

A \textit{lattice} is a free, finitely generated abelian group $L$ together with a non-degenerate integral symmetric bilinear form $b\colon L\times L\to \Z$. We usually write $x\cdot y\coloneqq b(x,y)$ and $x^2\coloneqq b(x,x)$ for $x,y\in L$. We assume all our lattices to be \textit{even}, that is, $x^2$ is even for all $x \in L$. 

For a vector $x\in L$, we call the positive integer 
\[
\dv(x)\coloneqq \gcd_{y\in L}(x\cdot y)
\]
the \textit{divisibility} of $x$. 
We denote by $L^* = \Hom(L,\Z)$ the \textit{dual} of $L$. There is a natural embedding $L\hookrightarrow L^*$ given by the bilinear form, which is injective because we assume $L$ to be non-degenerate. The quotient $A_{L}\coloneqq L^*/L$ is called the \textit{discriminant group} of $L$. This is a finite abelian group, and we call the order of $A_L$ the \textit{discriminant} of $L$: 
\[
\disc(L)\coloneqq |A_L|.
\]
We call $L$ \textit{unimodular} if $\disc(L) = 1$.
The discriminant group comes equipped with a natural quadratic form $q\colon A_L\to \Q/2\Z$. 

If $L$ is a lattice, we will write $L(-1)$ for the lattice with the same underlying group as $L$, but whose bilinear form is multiplied by $-1$. Similarly, we write $A_L(-1)$ for the finite quadratic module whose underlying group is $A_L$ and whose quadratic form is flipped by a sign.

For a primitive vector $x\in L$, the element $$\frac{x}{\dv(x)}$$ determines an element of $A_L$ of order $\dv(x)$. Therefore, 
\begin{equation}\label{eq: divisibility divides discriminant}
    \dv(x)\mid\disc(L).
\end{equation}

A morphism of lattices is called a \textit{metric morphism}, and an isomorphism of lattices is called an \textit{isometry}. Note that metric morphisms are injective, as we assume our lattices to be non-degenerate. A metric morphism $N\hookrightarrow L$ is called a \textit{primitive embedding} if the quotient $L/N$ is torsion-free.

\begin{lemma}\cite[Proposition 1.5.1]{Nik80} \label{lem: discriminants in unimodular lattices}
    Let $H$ be a lattice, and let $N\subset H$ be a primitive sublattice with orthogonal complement $T\coloneqq N^\perp \subset H$. Write $G\coloneqq \frac{H}{N\oplus T}$ for the gluing group. Then there are natural group homomorphisms
    \[
    \xymatrix{
    & G \ar[dl]_{i_N} \ar[dr]^{i_T}&\\
    A_N && A_T
    }
    \]
    which are injective. Moreover, if $H$ is unimodular, then $i_N$ and $i_T$ are group isomorphisms.
\end{lemma}
\begin{proof}
    We will prove the claims for $i_N$, and then the result for $i_T$ follows by the symmetry of the situation. For an element $x\in H$, write $\overline{x}$ for its image in $G$. The element $x$ defines a group homomorphism $(x\cdot -)|_N\colon N\to \Z$, which we view as an element of $N^*$, and thus it induces an element of $A_N$, which we shall call $i_N(\overline{x})$. The construction of $i_N(\overline{x})$ is readily seen to be independent of the choice of a lift $x$ of $\overline{x}$, so we obtain a group homomorphism $i_N\colon G\to A_N$. 
    
    If two elements $x,y\in H$ satisfy $i_N(\overline{x}) = i_N(\overline{y})$, then by construction there is an element $v\in N$ such that $(x\cdot -)|_N - (y\cdot -)|_N = (v\cdot -)$. Therefore, $x-y-v\in N^\perp = T$, so that we may write $x-y = v+ t$ for some $t\in T$. But this precisely means that $\overline{x} = \overline{y}$, so $i_N$ is injective. 

    Let now assume that $H$ is unimodular. Since $N\subset H$ is a primitive sublattice, we may extend any group homomorphism $f\colon N\to \Z$ to a group homomorphism $f_H\colon H\to \Z$. Since $H$ is unimodular, there exists an element $x\in H$ such that $(x\cdot -) = f_H$. In this case, $i_N(\overline{x})$ is equal to the class of $f$ in $A_N$, so $i_N$ is surjective and therefore an isomorphism. 
\end{proof}

In the notation of Lemma \ref{lem: discriminants in unimodular lattices}, the composition $$i_N(G) \overset{i_N^{-1}}{\longrightarrow} G \overset{i_T}{\longrightarrow} i_T(G)$$ is an anti-isometry, i.e. we have $i_N(G) \simeq i_T(G)(-1)$ \cite[Proposition 1.5.1]{Nik80}. 
Moreover, using this anti-isometry, we can recover $H$ via
\[
H \simeq \left\{ (x,y)\in N^*\oplus T^* \mid \overline{x}\in i_N(G) \text{ and } \gamma(\overline{x}) = \overline{y} \right\}.
\]
This explains why we call $G$ the gluing group. 

\begin{lemma}\label{lem: gluing subgroup order}
   Let $H$ be a lattice, and let $N\subset H$ be a primitive sublattice with orthogonal complement $T\coloneqq N^\perp \subset H$. Write $G\coloneqq \frac{H}{N\oplus T}$ for the gluing group. We have \[\disc(T)\cdot \disc(N) = |G|^2\cdot \disc(H).\]
\end{lemma}
\begin{proof}
    Consider the sequence of embeddings \[T\oplus N \overset{(1)}{\hookrightarrow} H \overset{(2)}{\hookrightarrow} H^* \overset{(3)}{\hookrightarrow} (T\oplus N)^*.\] The embeddings $(1)$ and $(3)$ each have index $|G|$, and embedding $(2)$ has index $\disc(H)$, thus the index of the composition of the three embeddings is $|G|^2\cdot \disc(H)$. From the isomorphism \[\frac{(T\oplus N)^*}{T\oplus N} = A_{T\oplus N} \simeq A_{T}\oplus A_{N},\] it follows that this must equal $\disc(T)\cdot \disc(N)$.
\end{proof}

A \textit{Hodge lattice} is a lattice with a Hodge structure. Isomorphisms of Hodge lattices that respect the lattice structures are called \textit{Hodge isometries}. If $S$ is a K3 surface, then $H^2(S,\Z)$ and $T(S)$ are both Hodge lattices of K3-type, where the bilinear form is given by the cup-product, see for example \cite[\S3.2]{Huy16}. As a lattice, $H^2(S,\Z)$ is isometric to the \textit{K3-lattice} $$\klat\coloneqq E_8(-1)^{\oplus 2} \oplus U^{\oplus 3},$$ c.f. \cite[Proposition 1.3.5]{Huy16}. Here, $E_8$ is the unique positive-definite even unimodular lattice of rank $8$, and $U$ is the \textit{hyperbolic plane}, which is the unique indefinite, unimodular lattice of rank $2$, see for example \cite[Chapter V]{Ser93}.
This means that $H^2(S,\Z)$ is a unimodular lattice. As a result of Lemma \ref{lem: discriminants in unimodular lattices}, there are group isomorphisms 
\[
A_{T(S)} \simeq G(S) \simeq A_{\NS(S)}(-1).
\]
The full cohomology group $H^*(S,\Z)$ is a Hodge lattice of K3-type, where $H^0(S,\Z)\oplus H^4(S,\Z)$ has the structure of a hyperbolic plane which is orthogonal to $H^2(S,\Z)$, and the Hodge structure is induced by the Hodge structure on $H^2(S,\Z)$. With this Hodge lattice structure, we denote the full cohomology group by $\tilde{H}(S,\Z)$ and call it the Mukai lattice of $S$ \cite[\S1]{Muk87}. The algebraic part of the Mukai lattice is called the extended N\'eron--Severi lattice, which we denote by $\N(S) \coloneqq H^0(S,\Z)\oplus \NS(S) \oplus H^4(S,\Z)\subset \tilde{H}(S,\Z)$. As a lattice, $\tilde{H}(S,\Z)$ is isometric to the lattice $$\extklat \coloneqq E_8(-1)^{\oplus 2}\oplus U^{\oplus 4}$$ which we call the extended K3-lattice. 

If $X$ is a hyperk\"ahler manifold of \Kn-type, where $n\geq 2$, then $H^2(X,\Z)$ is also a Hodge lattice of K3-type, where the lattice structure is given by the Beauville--Bogomolov--Fujiki form \cite{Bea83,Bog96,Fuj87}. The lattice structure of $H^2(X,\Z)$ is isometric to the lattice 
\begin{equation} \label{eq: k3n lattice}
E_8(-1)^{\oplus 2}\oplus U^{\oplus 3}\oplus \langle2-2n\rangle,
\end{equation}
where $\langle2-2n\rangle$ denotes the rank 1 lattice whose generator has square $2-2n$  \cite[\S 6]{Bea83}. Thus, $H^2(X,\Z)$ is not unimodular, since $\disc(H^2(X,\Z)) = 2n-2$. The gluing group $G(X)$ may be a non-trivial subgroup of $A_{T(X)}$ and of $A_{\NS(X)}$ in this case, by Lemma \ref{lem: discriminants in unimodular lattices} and Lemma \ref{lem: gluing subgroup order}.

\subsection{The coarseness of a hyperk\"ahler manifold}
We recall the Birational Torelli Theorem for hyperk\"ahler manifolds of \Kn-type. 
\begin{theorem}\cite[Corollary 9.9]{Mar11} \label{thm: birational torelli theorem}
    Let $X$ be a hyperk\"ahler manifold of \Kn-type. There is a natural $O(\extklat)$-orbit of embeddings $i_X\colon H^2(X,\Z)\hookrightarrow \extklat$. Moreover, two hyperk\"ahler manifolds $X$ and $Y$ of \Kn-type are birational if and only if there is a commutative diagram of metric morphisms
    \begin{equation}
        \xymatrix{
            H^2(X,\Z) \ar[r]^{i_X} \ar[d]^\simeq & \extklat \ar[d]^\simeq \\
            H^2(Y,\Z) \ar[r]^{i_Y} & \extklat
        }
    \end{equation}
    where the left-hand-side vertical isometry is a Hodge isometry. 
\end{theorem}
For a hyperk\"ahler manifold $X$, pick any embedding $i_X\colon H^2(X,\Z)\hookrightarrow \extklat$ in the natural $O(\extklat)$-orbit. Then $\extklat$ inherits a Hodge lattice structure of K3-type from $H^2(X,\Z)$, which we denote by \mmlatx. Explicitly, the Hodge structure on $\mmlatx$ is the one defined by $\mmlatx^{2,0} = i_X(H^{2,0}(X))$.  
Note that $\mmlatx$ is only well-defined up to Hodge isometries. This Hodge lattice is usually called the \textit{Markman--Mukai lattice}, see for example \cite{Add16}.

The orthogonal complement of $H^2(X,\Z)$ in $\mmlatx$ is a lattice of rank 1 generated by a vector $v\in \mmlatx^{1,1}$ of square $2n-2$. Indeed, we have $\rk (H^2(X,\Z)^\perp) = \rk (\mmlatx) - \rk (H^2(X,\Z)) = 1$. Moreover, the discriminant group of $H^2(X,\Z)$ is cyclic of order $2n-2$ by \eqref{eq: k3n lattice}. Since $\mmlatx$ is unimodular, the discriminant group of $H^2(X,\Z)^\perp$ is also cyclic of order $2n-2$, thus we have $v^2 = 2n-2$. 
Since $v$ depends on the choice of an embedding $H^2(X,\Z)\hookrightarrow \extklat$, it is well-defined up to Hodge isometries of $\mmlatx$.
\begin{definition}\label{def: coarseness}
    Let $X$ be a hyperk\"ahler manifold of \Kn-type, and let $v\in \mmlatx$ be a generator for the orthogonal complement of $H^2(X,\Z)$. The \textit{coarseness} of $X$ is the positive integer 
    \[
    \crs(X)\coloneqq \dv(v) = \gcd_{w\in \mmlatx^{1,1}}(v\cdot w).
    \]
\end{definition}
We now list some basic facts about coarseness.
\begin{lemma}\label{lem: basic results about coarseness}
    Let $X$ be a hyperk\"ahler manifold of \Kn-type, where $n\geq 2$.
    \begin{itemize}
        \item[(i)] We have $\crs(X)\mid \disc(T(X))$ and $\crs(X)\mid 2n-2$.
        \item[(ii)] If $Y$ is a hyperk\"ahler manifold which is birational to $X$, then $\crs(X) = \crs(Y)$.
        \item[(iii)] If $X = M(v)$ is a smooth moduli space of stable sheaves on some K3 surface $S$, where $v\in \N(S)$ is a primitive, effective Mukai vector, then we have $\crs(X) = \dv(v)$, where $\dv(v)$ is the divisibility of $v$ in $\N(S)$.
        \item[(iv)] If $X = M(v)$ is a smooth moduli space of stable sheaves on some K3 surface $S$, then $\crs(X)$ is equal to the order of the Brauer class $\alpha_X\in \Br(X)$ that obstructs the existence of a universal sheaf on $S\times X$.
    \end{itemize}
\end{lemma}
\begin{proof}
    \begin{itemize}
        \item[(i)] Since $\crs(X)$ is the divisibility of a primitive vector in $\mmlatx^{1,1}$, it follows from \eqref{eq: divisibility divides discriminant} that $\crs(X)\mid \disc(\mmlatx^{1,1})$. Since $\mmlatx^{1,1} = T(X)^\perp\subset \mmlatx$ and $\mmlatx$ is unimodular, we have $\disc(\mmlatx^{1,1}) = \disc(T(X))$ by \cite[Corollary 1.6.2]{Nik80}. Moreover, $\crs(X)\mid 2n-2$ follows from the fact that $\dv(v)$ divides $v^2=2n-2$.
        \item[(ii)] There exists a Hodge isometry $\Phi\colon\mmlatx\simeq \tilde{\Lambda}_Y$ which maps $H^2(X,\Z)$ to $H^2(Y,\Z)$ by the Birational Torelli Theorem \ref{thm: birational torelli theorem}. Let $v\in \mmlatx$ and $v'\in \tilde{\Lambda}_Y$ be generators for the orthogonal complements of $H^2(X,\Z)$ and $H^2(Y,\Z)$. Then $\Phi(v) = \pm v'$, hence $$\crs(X) = \dv(v) = \dv(\Phi(v)) = \dv(\pm v') = \dv(v') = \crs(Y).$$
        \item[(iii)] If $X = M(v)$, then there is a Hodge isometry $H^2(X,\Z)\simeq v^\perp\subset \tilde{H}(S,\Z)$ \cite[Main Theorem]{OGr97}, and the natural $O(\extklat)$-orbit of embeddings is given by the embedding $v^\perp\subset \tilde{H}(S,\Z)\simeq \extklat$ \cite[Example 9.6]{Mar11}. Therefore we have $\crs(X) = \dv(v)$.
        \item[(iv)] This follows from (iii) combined with \cite[Theorem 4.5]{MM24}.
    \end{itemize}
\end{proof}

By Lemma \ref{lem: basic results about coarseness}(iv), a smooth moduli space $M(v)$ of stable sheaves on a K3 surface is fine if and only if $\crs(M(v)) = 1$, which explains why we call $\crs(M(v))$ the coarseness of $M(v)$. 

\section{Counterexamples to the Kuznetsov--Shinder conjecture}

The main goal of this section is to prove the following theorem, which provides counterexamples to Conjecture \ref{conj: D implies L}. 

\begin{theorem}\label{thm: D equivalence does not imply L equivalence Main Text}
    Let $(S,H)$ be a primitively polarised K3 surface with $H^2 = 2g-2 \geq 2$, and assume that $S$ has Picard rank 1. Then the moduli spaces of stable sheaves $\olPic^0\coloneqq M(0,H,1-g)$ and $\olPic^{g-1} = M(0,H,0)$ are $D$-equivalent but not $L$-equivalent. 
\end{theorem}

The fact that $\olPic^0$ and $\olPic^{g-1}$ are derived equivalent was proved in \cite{ADM16}. Our contribution is proving that they are not $L$-equivalent. Note that $\olPic^{0}$ and $\olPic^{g-1}$ are never birational for $g>1$ by Lemma \ref{lem: basic results about coarseness}(ii). Indeed, we will see in the proof of Theorem \ref{thm: D equivalence does not imply L equivalence Main Text} that we have $\crs(\olPic^{0}) = g-1$ and $\crs(\olPic^{g-1}) = 2g-2$.

The idea of the proof of Theorem \ref{thm: D equivalence does not imply L equivalence Main Text} is to show that the orders of the gluing groups of $\olPic^0$ and $\olPic^{g-1}$ are not equal, and then apply Lemma \ref{lem: gluing subgroup order is L-invariant}. 
If $X = M(v)$ is a smooth $2n$-dimensional moduli space of stable sheaves on a K3 surface $S$, then the Hodge isometry $H^2(X,\Z) \simeq v^\perp$ of \cite[Main Theorem]{OGr97} restricts to a Hodge isometry of transcendental lattices $T(X) \simeq T(S)$. thus we have $\disc(T(X)) = \disc(T(S))$. Moreover, it follows from \eqref{eq: k3n lattice} that $\disc(H^2(X,\Z)) = 2n-2$. Therefore, to compute $|G(X)|$ using Lemma \ref{lem: gluing subgroup order}, we only need to compute $\disc(\NS(X))$.

\begin{lemma}\label{lem: discriminants of neron severi lattices}
    Let $X$ be a hyperk\"ahler manifold of \Kn-type, where $n\geq 2$. Then we have 
    \begin{equation} \label{eq: determinants of neron severi lattices}
        \disc(\NS(X)) = \frac{\disc(T(X))\cdot (2n-2)}{\crs(X)^2}.
    \end{equation}
    Here, $\crs(X)$ denotes the coarseness of $X$, see Definition \ref{def: coarseness}.
\end{lemma}
Note that both $\disc(T(X))$ and $2n-2$ are divisible by $\crs(X)$ by Lemma \ref{lem: basic results about coarseness}(i), therefore the right-hand-side of \eqref{eq: determinants of neron severi lattices} is an integer.
\begin{proof}
    Denote by $v\in \mmlatx^{1,1}$ a generator for the orthogonal complement of $H^2(X,\Z)$. We have $\NS(X)\simeq v^\perp \cap \mmlatx^{1,1}$. This induces a chain of finite-index embeddings $$\NS(X)\oplus \Z v \hookrightarrow \mmlatx^{1,1}\hookrightarrow (\mmlatx^{1,1})^*\hookrightarrow (\NS(X)\oplus \Z v)^*.$$ It follows from Lemma \ref{lem: gluing subgroup order} that 
    \begin{equation}\label{eq: discriminants without computation of m}
    \disc(\NS(X)) \cdot (2n-2) = m^2 \cdot \disc(\mmlatx^{1,1}),
    \end{equation}
    where $m$ is the index $[\mmlatx^{1,1}:\NS(X)\oplus \Z v]$. Here we used that $\disc(\Z v) = 2n-2$.
    To compute $m$, we mimic the proof of \cite[Lemma 9.10]{Bec22}. We have a commutative diagram:

    \begin{equation*}
        \xymatrix@=1em{
            &0\ar[d] &0 \ar[d] &&\\
            0 \ar[r] & \NS(X) \ar[r] \ar[d]& \NS(X) \ar[d]\ar[r]& 0 \ar[d] & \\ 
            0\ar[r] & \NS(X)\oplus \Z v \ar[r] \ar[d]& \mmlatx^{1,1} \ar[r] \ar[d] & G\ar[r] \ar[d]& 0 \\
            0 \ar[r]& \Z v \ar[r] \ar[d] & (\Z v)^* \ar[r] \ar[d] & A_v \ar[d] \ar[r] & 0\\
            & 0 \ar[r] & \Z/\dv(v)\Z \ar[r] \ar[d] & \Z/\dv(v)\Z \ar[r] \ar[d]& 0 \\
            &&0&0&
        }
    \end{equation*}
    Here, $A_v$ denotes the discriminant group of $\Z v$, which is cyclic of order $v^2 = 2n-2$. It follows that $$m=[\mmlatx^{1,1}:\NS(X)\oplus \Z v] = |G| = \frac{|A_v|}{\dv(v)} = \frac{2n-2}{\dv(v)}.$$ Combining this with \eqref{eq: discriminants without computation of m} gives the desired result. Here, we use the fact that $\disc(\mmlatx^{1,1}) = \disc(T(X))$, which holds because $\mmlatx^{1,1}$ is the orthogonal complement of $T(X)$ inside the unimodular lattice $\mmlatx$. 
\end{proof}

\begin{remark} \label{rem: discriminant of moduli spaces}
    If $X = M(v)$ is a $2n$-dimensional moduli space of stable sheaves on a K3 surface $S$, where $n\geq 2$, then \eqref{eq: determinants of neron severi lattices} becomes the following equation: 
    \begin{equation*}
        \disc(M(v)) = \frac{\disc(T(S))\cdot(2n-2)}{\dv(v)^2}.
    \end{equation*}
\end{remark}

\begin{theorem}\label{thm: L equivalent same divisibility}
    Let $X$ be a hyperk\"ahler manifold of \Kn-type, where $n\geq 2$. Then $$|G(X)| = \frac{\disc(T(X))}{\crs(X)}.$$ In particular, if $X$ is $L$-equivalent to a hyperk\"ahler manifold $Y$ of \Kn-type, then we have $$ \frac{\disc(T(X))}{\crs(X)} =  \frac{\disc(T(Y))}{\crs(Y)}.$$
\end{theorem}
\begin{proof}
    Lemma \ref{lem: gluing subgroup order} for $H = H^2(X,\Z)$ yields:
    \[
    \disc(T(X)) \cdot \disc(\NS(X)) = |G(X)|^2 \cdot (2n-2). 
    \]
    Combining this with Lemma \ref{lem: discriminants of neron severi lattices}, we obtain 
    \[
    |G(X)|^2 = \frac{\disc(T(X))^2}{\crs(X)^2}, 
    \] which shows that $|G(X)| = \frac{\disc(T(X))}{\crs(X)}.$

    The final claim follows from Lemma \ref{lem: gluing subgroup order is L-invariant}.
\end{proof}

\begin{corollary}\label{cor: L equivalent moduli spaces have same divisibility}
    Let $M = M(v)$ and $M'=M(u)$ be two smooth moduli spaces of stable sheaves, on the same K3 surface $S$, of dimension $2n\geq 4$. 
    Then we have 
    \begin{equation}\label{eq: gluing order for moduli spaces}
        |G(M(v))| = \frac{\disc(T(S))}{\dv(v)}.
    \end{equation}
    Thus, if $M$ and $M'$ are $L$-equivalent, then we have $\dv(v) = \dv(u)$. 
\end{corollary}
\begin{proof}
    Since there is a Hodge isometry $H^2(M,\Z) \simeq v^\perp\subset \tilde{H}(S,\Z)$ \cite{OGr97,Yos01}, we have a Hodge isometry $T(S)\simeq T(M)$. Moreover, it follows from Lemma \ref{lem: basic results about coarseness}(iv) that $\crs(M) = \dv(v)$. Therefore, \eqref{eq: gluing order for moduli spaces} follows from Theorem \ref{thm: L equivalent same divisibility}. The equality $\dv(v) = \dv(u)$ now follows from \eqref{eq: gluing order for moduli spaces} combined with Lemma \ref{lem: gluing subgroup order is L-invariant}.
\end{proof}

Theorem \ref{thm: D equivalence does not imply L equivalence Main Text} is an easy consequence of Corollary \ref{cor: L equivalent moduli spaces have same divisibility} combined with \cite{ADM16}:

\begin{proof}[Proof of Theorem \ref{thm: D equivalence does not imply L equivalence Main Text}]
    The moduli spaces $\olPic^0$ and $\olPic^{g-1}$ are derived equivalent by \cite[Proposition 2.1]{ADM16}. We now compute the coarseness of $\olPic^0$ and $\olPic^{g-1}$. By Lemma \ref{lem: basic results about coarseness}(iii), we have $\crs(\olPic^0) = \dv(0,H,1-g)$, where the divisibility is taken in the lattice $\N(S)$. Moreover, since $S$ has Picard rank 1, we find $\N(S) = H^0(S,\Z) \oplus \Z H \oplus H^4(S,\Z)$. We compute $$\dv(0,H,1-g) = \gcd(H^2,1-g) = \gcd(2g-2, 1-g) = g-1.$$ Similarly, we find $\crs(\olPic^{g-1}) = 2g-2$. Since $2g-2\geq 2$, it follows that $\crs(\olPic^{0}) = g-1 \neq 2g-2 = \crs(\olPic^{g-1})$, thus $\olPic^0$ and $\olPic^{g-1}$ are not $L$-equivalent by Theorem \ref{thm: L equivalent same divisibility}. 
\end{proof}

\begin{remark} \label{rem: it also works for K3 surfaces, sort of}
    If $n=1$, then \eqref{eq: gluing order for moduli spaces} needs to be modified. Indeed, if $M$ is a K3 surface, Then we have $|G(M)| = \disc(T(M))$, and there is a short exact sequence \cite{Muk87}
    \[
    0\to T(S) \to T(M) \to \Z/\dv(v)\Z \to 0, 
    \] from which it follows that $\disc(T(M)) = \frac{\disc(T(S))}{\dv(v)^2}$. The conclusion of Corollary \ref{cor: L equivalent moduli spaces have same divisibility} still holds, i.e. if $M(v)$ and $M(u)$ are 2-dimensional smooth moduli spaces of stable sheaves on $S$ which are $L$-equivalent, then $\dv(v) = \dv(u)$. 
    However, this does not lead to counterexamples to Conjecture \ref{conj: D implies L}, since $D$-equivalent moduli spaces $M(v),$  $M(u)$ of dimension $2$ always satisfy $\dv(v) = \dv(u)$ by the Derived Torelli Theorem for K3 surfaces \cite{Muk87,Orl03}.
\end{remark}

\begin{corollary}\label{cor: L equivalence preserves fineness}
    Let $M(v)$ and $M(u)$ be moduli spaces of sheaves on a K3 surface $S$, of dimension $2n \geq 2$. Assume that $M(v)$ is a fine moduli space, and that $M(v)$ and $M(u)$ are $L$-equivalent. Then $M(u)$ is a fine moduli space.   
\end{corollary}
\begin{proof}
    Since the moduli space $M(v)$ is fine, it follows that $\dv(v) =1$ by \cite[Theorem 5.3.1]{Cal00} and \cite[Theorem 4.5]{MM24}. Thus, by Corollary \ref{cor: L equivalent moduli spaces have same divisibility} and Remark \ref{rem: it also works for K3 surfaces, sort of}, we have $\dv(u)=1$ as well. In particular, $M(u)$ is also a fine moduli space.
\end{proof}

\printbibliography
\end{document}